\documentclass[12pt, oneside, a4paper]{article}

\usepackage{amsmath}
\usepackage{amsfonts}
\usepackage{amssymb}
\usepackage{amsthm,mathrsfs}
\usepackage{enumerate}
\usepackage{graphicx}
\newtheorem{theorem}{Theorem}[section]
\newtheorem{lemma}[theorem]{Lemma}

\theoremstyle{definition}

\title{\textbf{$n$-exact Character Graphs}}
\author{Mahdi Ebrahimi\footnote{ m.ebrahimi.math@ipm.ir}
 \\
 {\small\em  School of Mathematics, Institute for Research in Fundamental Sciences (IPM)},\\{\small\em P.O. Box: 19395--5746, Tehran, Iran}}

\date{}

\begin{document}

\maketitle

\begin{abstract}
Let $\Gamma$ be a finite simple graph. If for some integer $n\geqslant 4$, $\Gamma$ is a $K_n$-free graph whose complement has an odd cycle of length at least $2n-5$, then we say that $\Gamma$ is an $n$-exact graph. For a finite group $G$, let $\Delta(G)$ denote the character graph built on the set of degrees of the irreducible complex characters of $G$. In this paper, we prove that the order of an $n$-exact character graph is at most $2n-1$. Also
 we determine the structure of all
 finite groups $G$ with extremal $n$-exact character graph $\Delta(G)$.

 \end{abstract}

\section{Introduction}
$\noindent$ Let $G$ be a finite group and $R(G)$ be the solvable radical of $G$. Also let ${\rm cd}(G)$ be the set of all character degrees of $G$, that is,
 ${\rm cd}(G)=\{\chi(1)|\;\chi \in {\rm Irr}(G)\} $, where ${\rm Irr}(G)$ is the set of all complex irreducible characters of $G$. The set of prime divisors of character degrees of $G$ is denoted by $\rho(G)$. It is well known that the
 character degree set ${\rm cd}(G)$ may be used to provide information on the structure of the group $G$. For example, a result due to Berkovich \cite{[B]} says that if a prime $p$ divides
 every non-linear character degree of a group $G$, then $G$ is solvable.

A useful way to study the character degree set of a finite group $G$ is to associate a graph to ${\rm cd}(G)$.
One of these graphs is the character graph $\Delta(G)$ of $G$ \cite{[I]}. Its vertex set is $\rho(G)$ and two vertices $p$ and $q$ are joined by an edge if the product $pq$ divides some character degree of $G$. We refer the readers to a survey by Lewis \cite{[M]} for results concerning this graph and related topics.

There is an interesting conjecture on the structure of $\Delta(G)$ which is posed by Akhlaghi and Tong-Viet in \cite{[AT]}:\\

\noindent\textbf{Conjecture.} Let $G$ be a finite group such that for some positive integer $n$, $\Delta(G)$ is $K_n$-free. If $G$ is solvable, then $|\rho(G)|\leqslant 2n-2$, and if $G$ is non-solvable, then $|\rho(G)|\leqslant 2n-1$.\\
In \cite{1}, it was shown that for a solvable group $G$, the complement of $\Delta(G)$ does not have an odd cycle, in other word, it is bipartite. This result guarantees this conjecture for solvable groups. The conjecture still is not proved for non-solvable groups, however the results in \cite{[Ton]}, \cite{[AT]} and \cite{kh} say that it is true when $n\in\{3,4,5\}$.
Now let $\Gamma$ be a finite simple graph. If for some integer $n\geqslant 4$, $\Gamma$ is a $K_n$-free graph whose complement has an odd cycle of length at least $2n-5$, then we say that $\Gamma$ is an $n$-exact graph. For example, $K_4$-free graphs with non-bipartite complement are $4$-exact. In this paper, we wish to show that the conjecture is true for a finite group whose character graph is $n$-exact, for some integer $n\geqslant 4$.\\

\noindent\textbf{Theorem A.}
 \textit{Let $G$ be a finite group, and $n\geqslant 4$ be an integer. If $\Delta(G)$ is an $n$-exact graph, then there exists a normal subgroup $R(G)<M\leqslant G$ so that $G/R(G)$ is an almost simple group with socle $S:=M/R(G)\cong \rm{PSL}_2(q)$, where $q$ is a prime power, and $|\rho(G)|\leqslant 2n-1$. }\\

  The solvable group $G$ is said to be disconnected if $\Delta(G)$ is disconnected. Also $G$ is called of disconnected Type $n$ if $G$ satisfies the hypotheses of Example $2.n$ in \cite{los}. Let $n\geqslant 4$ be an integer and $\Gamma$ be a finite simple graph with vertex set $V$. The cardinality of $V$ is called the order of $\Gamma$. An $n$-exact graph $\Gamma$ with minimum or maximum possible order is called an extremal $n$-exact graph.
   Note that for an integer $n\geqslant 1$,  the set of prime divisors of $n$ is denoted by $\pi(n)$. Now we are ready to state the next result of this paper.\\

\noindent \textbf{Corollary B.}  \textit{Let $G$ be a finite group with extremal $n$-exact character graph $\Delta(G)$, for some integer $n\geqslant 4$. Then one of the following cases occurs:\\
\textbf{a)}  $|\rho(G)|=2n-5$  and for some integer $\alpha\geqslant 2$, $G\cong \rm{PSL}_2(2^\alpha)\times R(G)$, where $|\pi(2^\alpha\pm 1)|=n-3$ and $R(G)$ is abelian. \\
\textbf{b)}  $|\rho(G)|=2n-1$ and for some integer $\alpha\geqslant 2$, $G\cong \rm{PSL}_2(2^\alpha)\times R(G)$, where $|\pi(2^\alpha\pm 1)|=n-3,n-2$ or $n-1$. Also\\
\textbf{i)} If $|\pi(2^\alpha\pm 1)|=n-3$, then for some disconnected groups $A$ and $B$ of disconnected Types $1$ or $4$, $R(G)\cong A\times B$, $|\rho (A)|=|\rho (B)|=2$ and $\rho(G)=\pi (S)\uplus \rho(A)\uplus \rho(B)$.\\
\textbf{ii)} If $|\pi(2^\alpha\pm 1)|=n-2$, then  $R(G)$ is a disconnected group of disconnected Type $1$ or $4$, $|\rho(R(G))|=2$ and $\rho(G)=\pi (S)\uplus \rho(R(G))$.\\
\textbf{iii)} If $|\pi(2^\alpha\pm 1)|=n-1$, then  $R(G)$ is abelian and $\rho(G)=\pi(S)$.
}

\section{Preliminaries}
$\noindent$ In this paper, all groups are assumed to be finite and all
graphs are simple and finite. For a finite group $G$, the set of prime divisors of $|G|$ is denoted by $\pi(G)$.
If $H\leqslant G$ and $\theta \in \rm{Irr}(H)$, we denote by $\rm{Irr}(G|\theta)$ the set of irreducible characters of $G$ lying over $\theta$ and define $\rm{cd}(G|\theta):=\{\chi(1)|\,\chi \in \rm{Irr}(G|\theta)\}$. We frequently use,  Gallagher's Theorem which is corollary 6.17 of \cite{[isa]}.

\begin{lemma}\label{fraction}
Let $ N \lhd G$ and $\varphi \in \rm{Irr}(N)$. Then for every $\chi \in \rm{Irr}(G|\varphi)$, $\chi(1)/\varphi(1)$ divides $[G:N]$.
\end{lemma}

Let $\Gamma$ be a graph with vertex set $V(\Gamma)$ and edge set
$E(\Gamma)$. The complement of $\Gamma$ and the induced subgraph of $\Gamma$ on $X\subseteq V(\Gamma)$
 are denoted by $\Gamma^c$ and  $\Gamma[X]$, respectively.  If $E(\Gamma)=\emptyset$, $\Gamma$ is called an empty graph.  When $n:=|V(\Gamma)|$, any cycle of
$\Gamma$ of length $n$ is called a Hamilton cycle. We say that $\Gamma$ is
Hamiltonian if it contains a Hamilton cycle.
We use the notations $K_n$ for a complete graph with $n$ vertices and $C_n$ for a cycle of length $n$. If for some integer $n\geqslant 2$, $\Gamma$ does not contain a copy  of $K_n$ as an induced subgraph, then $\Gamma$ is called a $K_n$-free graph. We now state some relevant results on character graphs
needed in the next sections.

\begin{lemma}\label{bipartite}\cite{1}
Let $G$ be a solvable group. Then $\Delta(G)^c$ is bipartite.
\end{lemma}

\begin{lemma}\label{square2}\cite{Sq}
Let $G$ be a solvable group with $\Delta(G)\cong C_4$. Then $G\cong A\times B$, where  $A$ and $B$ are disconnected groups. Also for some distinct primes $p,q,r$ and $s$, $\rho(A)=\{p,q\}$ and $\rho(B)=\{r,s\}$.
\end{lemma}

\begin{lemma}\label{square}\cite{Sq}
Let $G$ be a solvable group. If $\Delta(G)$ has at least $4$ vertices, then either $\Delta(G)$ contains a triangle or $\Delta(G)\cong C_4$.
\end{lemma}

The structure of the character graphs of $\rm{PSL}_2(q)$ and $^2B_2(q^2)$ are determined as follows:

\begin{lemma}\label{chpsl}\cite{[white]}
Let $G\cong \rm{PSL}_2(q)$, where $q\geqslant 4$ is a power of a prime $p$.\\
\textbf{a)}
 If $q$ is even, then $\Delta(G)$ has three connected components, $\{2\}$, $\pi(q-1)$ and $\pi(q+1)$, and each component is a complete graph.\\
\textbf{b)}
 If $q>5$ is odd, then $\Delta(G)$ has two connected components, $\{p\}$ and $\pi((q-1)(q+1))$.\\
i)
 The connected component $\pi((q-1)(q+1))$ is a complete graph if and only if $q-1$ or $q+1$ is a power of $2$.\\
ii)
  If neither of $q-1$ or $q+1$  is a power of $2$, then $\pi((q-1)(q+1))$ can be partitioned as $\{2\}\cup M \cup P$, where $M=\pi (q-1)-\{2\}$ and $P=\pi(q+1)-\{2\}$ are both non-empty sets. The subgraph of $\Delta(G)$ corresponding to each of the subsets $M$, $P$ is complete, all primes are adjacent to $2$, and no prime in $M$ is adjacent to any prime in $P$.
 \end{lemma}

  \begin{lemma}\label{chsz}\cite{[white1]}
Let $G\cong ^2B_2(q^2)$, where $q^2=2^{2m+1}$ and $m\geqslant 1$.
The set of primes $\rho(G)$ can be partitioned as $\rho(G)=\{2\}\cup \pi (q^2-1)\cup \pi (q^4+1)$.
The subgraph of $\Delta(G)$ on $\rho(G)-\{2\}$ is complete and $2$ is adjacent in $\Delta(G)$ to precisely the primes in $\pi(q^2-1)$.
  \end{lemma}

When $\Delta(G)^c$ is not a bipartite graph, then there exists a useful restriction on the structure of $G$ as follows:

\begin{lemma}\label{cycle} \cite{AC}
Let $G$ be a finite group and $\pi$ be a subset of the vertex set of $\Delta(G)$ such that $|\pi|> 1$ is an odd number. Then $\pi$ is the set of vertices of a cycle in $\Delta(G)^c$ if and only if $O^{\pi^\prime}(G)=S\times A$, where $A$ is abelian, $S\cong \rm{SL}_2(u^\alpha)$ or $S\cong \rm{PSL}_2(u^\alpha)$ for a prime $u\in \pi$ and a positive integer $\alpha$, and the primes in $\pi - \{u\}$ are alternately odd divisors of $u^\alpha+1$ and  $u^\alpha-1$.
\end{lemma}

  \begin{lemma}\label{lw}\cite{[non]}
 Let $p$ be a prime, $f\geqslant 2$ be an integer, $q=p^f\geqslant 5$ and $S\cong \rm{PSL}_2(q)$. If $q\neq9$ and $S\leqslant G\leqslant \rm{Aut}(S)$, then $G$ has irreducible characters of degrees $(q+1)[G:G\cap \rm{PGL}_2(q)]$ and $(q-1)[G:G\cap \rm{PGL}_2(q)]$.
 \end{lemma}

Now we present some facts on the structure of character graphs with non-bipartite complement.

 \begin{lemma}\label{ote}\cite{ME}
Let $G$ be a finite group, $R(G)< M\leqslant G$, $S:=M/R(G)$ be isomorphic to $ \rm{PSL}_2(q)$, where for some prime $p$ and positive integer $f\geqslant 1$, $q=p^f$, $|\pi(S)|\geqslant 4$ and $S\leqslant G/R(G)\leqslant \rm{Aut}(S)$. Also let $\theta\in \rm{Irr}(R(G))$. If $\Delta(G)^c$ is not  bipartite,  then $\theta$ is $M$-invariant.
\end{lemma}

\begin{lemma}\label{direct product}\cite{ME}
Assume that $f\geqslant 2$ is an integer, $q=2^f$,  $S\cong \rm{PSL}_2(q)$ and $G$ is a finite group such that $G/R(G)=S$. If $\Delta(G)[\pi(S)]=\Delta(S)$, then $G\cong S\times R(G)$.
\end{lemma}

\begin{lemma}\label{Hamilton}\cite{Me}
 Let $G$ be a finite group. The complement of the character graph $\Delta(G)$ is a non-bipartite Hamiltonian graph if and only if $G\cong \rm{SL}_2(2^f)\times A$, where $f\geqslant 2$ is an integer, $||\pi(2^f+1)|-|\pi(2^f-1)||\leqslant 1$ and $A$ is an abelian group.
 \end{lemma}

We end this section with the following result.

 \begin{lemma}\label{type} \cite{ME}
  Let $G$ be a disconnected group. If $|\rho(G)|=2$ and $2 \notin \rho(G)$, then $G$ is of disconnected Type $1$ or $4$.
  \end{lemma}


\section{The proof of Main results}
$\noindent$ In this section, we wish to prove our main results. When $G$ is a finite group with $4$-exact character graph, then using \cite{[AT]}, \cite{ME} and Lemma \ref{Hamilton} we have nothing to prove. Hence in the sequel, we assume that $G$ is a finite group such that for some integer $n\geqslant 5$, $\Delta(G)$ is $n$-exact.

\begin{lemma}\label{free}
Suppose $G$ is a finite group and $\Delta(G)$ is an $n$-exact graph, for some integer $n\geqslant 5$. Then \\
\textbf{a)}
There exists a normal subgroup $R(G)<M\leqslant G$ so that $G/R(G)$ is an almost simple group with socle $S:=M/R(G)\cong \rm{PSL}_2(u^\alpha)$, where $u$ is a prime, $\alpha$ is a positive integer and $u^\alpha > 5$.\\
\textbf{b)}
For some $\pi_0\subseteq \pi(S)$ and integer $m\geqslant n-3$, $\Delta(G)[\pi_0]$ precisely contains three connected components $\Delta(G)[\{u\}]$ and $\Delta(G)[\pi(u^\alpha \pm 1)\cap \pi_0]\cong K_{m}$.\\
 \textbf{c)}
 If $\pi$ is the set of vertices of an odd cycle in $\Delta(G)^c$, then  $u\in\pi\subseteq \pi (S)$.\\
 Now let $\pi_R:=\rho(R(G))-\pi(G/R(G))$, $\pi_M:=\pi([G:M])-\rho(M)$ and $\pi_F=\pi_R \cup \pi_M$.\\
  \textbf{d) }$\Delta(G)[\pi_M]$ is a complete graph with at most two vertices and every $p\in \pi_F$ is adjacent to all vertices in $\pi(u^{2\alpha}-1)$.\\
\textbf{e) } If $u=2$ and $\pi([G:M])\subseteq \pi_F$, then for every $\theta\in\rm{Irr}(R(G))$, $\Delta(G)[\pi(\theta(1)[G:M](u^\alpha\pm 1))]$ are complete.
\end{lemma}

\begin{proof}
\textbf{a,b)}
 Since $\Delta(G)$ is an $n$-exact graph, for some $\pi_0\subseteq \rho(G)$, $\pi_0$ is the set of vertices of an odd cycle in $\Delta(G)^c$ and $|\pi_0|\geqslant 2n-5$. By Lemma \ref{cycle}, $N:=O^{\pi^\prime_0}(G)=R\times A$, where $A$ is abelian, $R\cong \rm{SL}_2(u^\alpha)$ or $R\cong \rm{PSL}_2(u^\alpha)$ for a prime $u\in \pi_0$ and a positive integer $\alpha$,
 and the primes in $\pi_0 - \{u\}$ are alternately odd divisors of $u^\alpha+1$ and  $u^\alpha-1$. Hence $m:=|\pi(u^\alpha\pm 1)\cap \pi_0|\geqslant n-3$ and $\Delta(G)[\pi(u^\alpha\pm 1)\cap \pi_0]\cong K_{m}$. Assume that for some $\epsilon\in \{\pm 1\}$, there exist $x\in (\pi(u^\alpha+\epsilon)\cap \pi_0)\cup \{u\}$ and  $y\in \pi(u^\alpha-\epsilon)\cap \pi_0$ so that $x$ and $y$ are adjacent vertices in $\Delta(G)$. Then for some $\chi \in \rm{Irr}(G)$, $xy$ divides $\chi(1)$. No let $\theta \in \rm{Irr}(N)$ be a constituent of $\chi_N$. By Lemma \ref{fraction}, $xy$ divides $\theta(1)$. It is a contradiction. Thus $\Delta(G)[\{u\}]$ and $\Delta(G)[\pi(u^\alpha \pm 1)\cap \pi_0]$ are precisely connected components of $\Delta(G)[\pi_0]$.
  Let $M:=NR(G)$. Then $S:=M/R(G)\cong N/R(N)\cong \rm{PSL}_2(u^\alpha)$ is a non-abelian minimal normal subgroup of $G/R(G)$.

 Let $C/R(G)=C_{G/R(G)}(M/R(G))$. We claim that $C=R(G)$ and $G/R(G)$ is an almost simple group with socle $S=M/R(G)$. Suppose on the contrary that $C\neq R(G)$ and let $L/R(G)$ be a chief factor of $G$ with $L\leqslant C$. Then $L/R(G)\cong T^k$, for some non-abelian simple group $T$ and some integer $k\geqslant 1$.
 As $L\leqslant C$, $LM/R(G)\cong L/R(G)\times M/R(G) \cong S\times T^k$. Let $F:=\pi (S)\cap \pi(T)$. Note that $\pi_0\cap \rho (T)=\emptyset$, $|\rho(T)|\geqslant 3$ and $2\in F$. Now one of the following cases occurs:\\
 Case 1. $|F|=1$: Then $F=\{2\}$. Hence $3$ does not divide $|T|$. Therefore by \cite{suzuki},  $T$ is a Suzuki simple group. By Lemma \ref{chsz}, $\Delta(T)\subseteq \Delta(S\times T^k)$ contains a triangle. Thus $\Delta(S\times T^k)[\pi(u^\alpha +1)\cup \rho(T)]$ contains a copy of $K_n$ and it is a contradiction.\\
Case 2. $|F|\geqslant 2$: Then it is easy to see that $\Delta(S\times T^k)[\pi(u^\alpha +1)\cup \rho(T)]$ contains a copy of $K_n$ and it is again a contradiction.\\
\textbf{c)}
 Let $\pi\subseteq \rho(G)$ be the set of vertices of an odd cycle in $\Delta(G)^c$. Then using Lemma \ref{cycle}, $N_1:=O^{\pi^\prime}(G)=R_1\times A_1$, where $A_1$ is abelian, $R_1\cong \rm{SL}_2(u_1^{\alpha_1})$ or $R_1\cong \rm{PSL}_2(u_1^{\alpha_1})$ for a prime $u_1\in \pi$ and a positive integer $\alpha_1$,
 and the primes in $\pi - \{u_1\}$ are alternately odd divisors of $u_1^{\alpha_1}+1$ and  $u_1^{\alpha_1}-1$. Thus $S_1:=N_1R(G)/R(G)\cong \rm{PSL}_2(u_1^{\alpha_1})$ is a minimal normal subgroup of $G/R(G)$ and $u_1\in \pi\subseteq \pi(S_1)$. Hence by part (a), $S=S_1$ and $u=u_1$.\\
 \textbf{d) }
  Using Lemma \ref{lw}, Lemma \ref{ote} and this fact that $\rm{SL}_2(u^\alpha)$ is the schur representation of $S$,  $\Delta(G)[\pi_M]$ is complete and every $p\in \pi_F$ is adjacent to all vertices in $\pi(u^{2\alpha}-1)$. If $|\pi_M|\geqslant 3$, then $\Delta(G)[\pi_M\cup \pi(u^\alpha+1)]$ contains a copy of $K_n$ and it is a contradiction. Thus $|\pi_M|\leqslant 2$. \\
  \textbf{e) }
Let $d=[G:M]$. Then $G/R(G)\cong S \rtimes \mathbb{Z}_d$. Note that $\pi(d) \cap \pi(S)=\emptyset$. Therefore as the schur multiplier of $S$ is trivial, the schur multiplier of $G/R(G)$ is too. Hence by Lemma \ref{lw} and Gallagher's Theorem, for every $\theta\in \rm{Irr}(R(G))$, $\theta(1)d(u^\alpha\pm 1)\in\rm{cd}(G)$ and it completes the proof.
\end{proof}

\noindent\textit{\bf Proof of Theorem A.}
By Lemma \ref{free} (a,b), there exists a normal subgroup $R(G)<M\leqslant G$ so that $G/R(G)$ is an almost simple group with socle $S:=M/R(G)\cong \rm{PSL}_2(u^\alpha)$, where $u$ is a prime, $\alpha$ is a positive integer and $u^\alpha > 5$. Also for some $\pi_0\subseteq \pi(S)$ and integer $m\geqslant n-3$, $\Delta(G)[\pi_0]$ precisely contains three connected components $\Delta(G)[\{u\}]$ and $\Delta(G)[\pi(u^\alpha\pm 1)\cap \pi_0]\cong K_{m}$.
 It is easy to see that $\rho(G)=\pi_F\uplus \pi(S)$. Using Lemma \ref{free} (d), $\Delta(G)[\pi_M]$ is a complete graph with at most two vertices and every $p\in \pi_F$ is adjacent to all vertices in $\pi(u^{2\alpha}-1)$.              Now one of the following cases occurs:\\
 Case 1: $|\pi(u^\alpha\pm 1)|=n-3$. Then by Lemma \ref{chpsl}, $u=2$ and $\Delta(G)[\pi_0]=\Delta(S)$. We claim that $|\pi_F|\leqslant 4$. On the contrary, we assume that $|\pi_F|\geqslant 5$. Let $d:=[G:M]$. Then  using Lemma \ref{lw}, $\pi(d) \cap \pi(S)=\emptyset$ and $\pi(d)\subseteq \pi_F$.
   If $\pi_M\neq \emptyset$, then as $|\pi_M|\leqslant 2$, using Lemmas \ref{bipartite} and \ref{free} (e), $\Delta(G)[\pi_F\cup \pi(u^\alpha+1)]$ contains a copy of $K_n$ which is impossible. Thus $\pi_M=\emptyset$ and $\pi_F=\pi_R$. Therefore using Lemma \ref{bipartite}, $\Delta(G)[\pi_F]$ contains a triangle and by Lemma \ref{free} (e), we again obtain a contradiction. Thus $|\pi_F|\leqslant 4$ and $|\rho(G)|\leqslant 2n-1$.\\
  Case 2: $|\pi(u^\alpha\pm 1)|\leqslant n-2$ and for some $\epsilon \in \{\pm 1\}$, $|\pi(u^\alpha+ \epsilon)|=n-2$. If $|\pi_F|\geqslant 3$, then using Lemma \ref{free} (c), $\Delta(G)[\pi_F]$ is a non-empty graph and  $\Delta(G)[\pi_F\cup\pi(u^\alpha+ \epsilon) ]$ contains a copy of $K_n$ which is a contradiction. Thus $|\pi_F|\leqslant 2$ and $|\rho(G)|\leqslant 2n-1$.\\
   Case 3: For some $\epsilon \in \{\pm 1\}$, $|\pi(u^\alpha+ \epsilon)|= n-1$. Then by Lemma \ref{free} (d), $\pi_F=\emptyset$ and so $|\rho(G)|\leqslant 2n-1$.\qed

   \subsection{Proof of Corollary B.}

\noindent Let $G$ be a finite group with extremal $n$-exact character graph $\Delta(G)$, for some integer $n\geqslant 5$.
Then using Theorem A, $|\rho(G)|=2n-5$ or $2n-1$. If $|\rho(G)|=2n-5$, then by Lemma \ref{Hamilton}, we are done. Thus we can assume that  $|\rho(G)|=2n-1$. By Lemma \ref{free}, there exists a normal subgroup $R(G)<M\leqslant G$ so that $G/R(G)$ is an almost simple group with socle $S:=M/R(G)\cong \rm{PSL}_2(u^\alpha)$, where $u$ is a prime, $\alpha$ is a positive integer and $u^\alpha > 5$.
 Also for some $\pi_0\subseteq \pi(S)$ and integer $m\geqslant n-3$, $\Delta(G)[\pi_0]$ precisely contains three connected components $\Delta(G)[\{u\}]$  and  $\Delta(G)[\pi(u^\alpha\pm 1)\cap \pi_0]\cong K_{m}$.
Since $  \Delta(G)$ is $K_n$-free, $\Delta(S)$ is too. The structure of $\Delta(S)$ is determined by  Lemma \ref{chpsl}.

\begin{lemma}\label{easy}
Let $\Delta(G)[\pi_F]$ be a non-empty graph and $|\pi(S)|\geqslant 2n-4$. Then $G$ does not exist.
\end{lemma}

\begin{proof}
Since $|\pi(S)|\geqslant 2n-4$, there exists $\epsilon \in \{\pm 1\}$ such that $|\pi(u^\alpha+\epsilon)|\geqslant n-2$. Thus using Lemma \ref{free} (d), $\Delta(G)[\pi_F\cup \pi(u^\alpha+\epsilon)]$ contains a copy of $K_n$ and it is a contradiction.
\end{proof}

\begin{lemma}\label{threevertex}
If $|\pi(u^\alpha\pm 1)|=n-3$, then $u=2$ and $G\cong S\times A\times B$, where $A$ and $B$ are disconnected groups of disconnected Types $1$ or $4$, $|\rho(A)|=|\rho(B)|=2$ and $\rho(G)=\pi(S)\uplus \rho(A)\uplus \rho(B)$.
\end{lemma}

\begin{proof}
Since $|\pi(u^\alpha\pm 1)|=n-3$, $\pi(S)=\pi_0$. Thus $|\pi_F|=4$, $u=2$ and $\Delta(G)[\pi(S)]=\Delta(S)$. If $d:=[G:M]$, then using Lemma \ref{lw}, $\pi(d) \cap \pi(S)=\emptyset$ and so $\pi(d)\subseteq \pi_F$. Now we claim that $d=1$ and $G=M$. On the contrary, we assume that $d\neq 1$. If $|\pi(d)|\geqslant 2$, then by Lemma \ref{free} (e), $\Delta(G)[\pi_F \cup \pi(u^\alpha+1)]$ contains a copy of $K_n$ which is a contradiction. Also if $|\pi(d)|=1$, then using Lemma \ref{bipartite}, $\Delta(G)[\pi_R]$ is non-empty and by Lemma \ref{free} (e), $\Delta(G)[\pi_F \cup \pi(u^\alpha+1)]$ again contains a copy of $K_n$ and it is a contradiction. Therefore $d=1$ and $G=M$.
 Hence by Lemma \ref{direct product},  $G\cong  S\times R(G)$. Thus as $\Delta(G)[\pi(S)]=\Delta(S)$, $\pi(S)\cap \rho(R(G))=\emptyset$ and $\rho(R(G))=\pi_R=\pi_F$. Also as $\Delta(G)$ is $K_n$-free, $\Delta(R(G))$ is $K_3$-free. Hence  using Lemma \ref{square}, $\Delta(R(G))\cong C_4$. Thus by Lemma \ref{square2}, $R(G)\cong A\times B$, where $A$ and $B$ are disconnected groups with $|\rho(A)|=|\rho(B)|=2$. Therefor by Lemma \ref{type},  $A$ and $B$ are of disconnected Types $1$ or $4$ and the proof is completed.
\end{proof}

\begin{lemma}\label{two7}
If  $|\pi(u^\alpha\pm 1)|=n-2$ and $u=2$, then $G\cong S\times R(G)$, where
 $R(G)$ is a disconnected  group of disconnected Type $1$ or $4$. Also $|\rho(R(G))|=2$ and $\rho(G)=\pi(S) \uplus \rho(R(G))$.
\end{lemma}

\begin{proof}
Since $|\rho(G)|=2n-1$, $|\pi_F|=2$. If $\pi_F=\pi_M$, then using Lemma \ref{free} (d), $\Delta(G)[\pi_F \cup \pi(u^\alpha+1)]$ is a copy of $K_n$ which is a contradiction.
  Hence $\pi_R$ is non-empty. Note that there exist $x\in \pi(u^\alpha+1)$ and $y\in \pi(u^\alpha-1)$ such that $\pi(u^\alpha+1)= \pi_0 \cup \{x\}$ and $\pi(u^\alpha-1)= \pi_0 \cup \{y\}$. Let $d:=[G:M]$. We claim that $x,y\notin \pi(d)$. On the contrary, we assume that $z\in \{x,y\} \cap \pi(d)$. Then using Lemmas \ref{lw} and \ref{free} (d), $\Delta(G)[\pi_R \cup \{z\} \cup \pi(u^{2\alpha}-1)]$ contains a copy of $K_n$ which is impossible. Thus $x,y\notin \pi(d)$.
   Suppose  $p\in \pi_0\cap \pi(d)$. Then by Lemma \ref{lw}, $p$ is adjacent to all vertices in $\pi_0-\{u,p\}$. It is a contradiction with the structure of $\Delta(G[\pi_0])$.
   Therefore $\pi(d)\subseteq \pi_F$.  If $d\neq 1$, then by Lemma \ref{free} (e), $\Delta(G)[\pi_F \cup \pi(u^\alpha+1)]$ contains a copy of $K_n$ which is a contradiction. Thus $d=1$ and $G=M$.
Hence there exist distinct primes $q$ and $q^\prime$ such that $\pi_F=\pi_R=\{q,q^\prime\}$. Note that as the schur multiplier of $S$ is trivial, by Gallagher's Theorem, $q$ and $q^\prime$ are adjacent to all vertices in $\pi(S)$.  Now we claim that $\Delta(G)[\pi(S)]=\Delta(S)$. On the contrary, we assume that there exist adjacent vertices $a,b \in \pi(S)$ in $\Delta(G)$ such that $a$ and $b$ are non-adjacent vertices in $\Delta(S)$.
 Then for some $\chi \in \rm{Irr}(G)$, $ab$ divides $\chi(1)$. Let ${\theta} \in \rm{Irr}(R(G))$ be a constituent of $\chi_{R(G)}$. Then using  Lemma \ref{ote},
 ${\theta}$ is $G$-invariant. Thus as the Schur multiplier of $S$ is trivial, by Gallagher's Theorem, $\rm{cd}(G|{\theta})=\{m{\theta}(1)|\,m\in \rm{cd}(S)\}$. Hence as $\chi(1)\in \rm{cd}(G|{\theta})$, ${\theta}(1)$ is divisible by $a$ or $b$. Without loss of generality, we assume that $a$ divides ${\theta}(1)$. There exists $\epsilon\in \{\pm 1\}$ so that $a\notin \pi(u^\alpha+\epsilon)$. Thus the induced subgraph of $\Delta(G)$ on $X:=\pi(u^\alpha+\epsilon)\cup\{a,q\}$ is a copy of $K_n$ and it is a contradiction.
Therefore $\Delta(G)[\pi(S)]=\Delta(S)$. Hence by Lemma \ref{direct product}, $G \cong S\times R(G)$. If $\rho (R(G))\cap \pi(S)\neq \emptyset$, then the induced subgraph of $\Delta(G)$ on $\pi(S)\cup \{q\}$ contains  a copy of $K_n$ and it is a contradiction. Thus as $\Delta(G)$ is $K_n$-free, it is easy to see that $\rho(R(G))=\pi_R=\{q,q^\prime\}$ and $\Delta(R(G))\cong K_2^c$. Therefore by Lemma \ref{type}, $R(G)$ is a disconnected group of disconnected Type $1$ or $4$. Finally, $\rho(G)=\pi(S)\uplus \rho(R(G))$ and the proof is completed.
\end{proof}

\begin{lemma}\label{five7}
If $|\pi(u^\alpha\pm 1)|=n-1$ and $u=2$, then $G\cong S\times R(G)$, where $R(G)$ is abelian and $\rho(G)=\pi(S)$.
 \end{lemma}

\begin{proof}
Since $|\rho(G)|=2n-1$, $\rho(G)=\pi(S)$.
If $G \neq M$, then using Lemma \ref{lw}, $\Delta(G)$ contains a copy of $K_n$ which is impossible.
  Thus $G=M$. We claim that $\Delta(G)[\pi(S)]=\Delta(S)$. On the contrary, we assume that there exist non-adjacent vertices $x$ and $y$ in $\Delta(S)$ such that $x$ and $y$ are adjacent vertices in $\Delta(G)$. Thus for some $\chi\in \rm{Irr}(G)$, $\chi(1)$ is divisible by $xy$. Let $\theta\in \rm{Irr}(R(G))$ be a constituent of $\chi_{R(G)}$. Using Lemma \ref{ote}, $\theta$ is $G$-invariant. Thus as the Schur multiplier of $S$ is trivial, $\theta$ extends to $G$. By Gallagher's Theorem, $\rm{cd}(G|\theta)=\{m\theta(1)|\;m\in \rm{cd}(S)\}$.  Thus as $\chi(1)\in \rm{cd}(G|\theta)$ is divisible by $xy$, $\theta(1)$ is divisible by $x$ or $y$. Hence one of the character degrees $\theta(1)(u^\alpha-1)$ or $\theta(1)(u^\alpha+1)$ of $G$ is divisible by $n$ distinct primes which is a contradiction.
 Therefore $\Delta(G)[\pi(S)]=\Delta(S)$. Thus using Lemma \ref{direct product}, $G\cong S\times R(G)$.
Hence as $\Delta(G)$ is $K_n$-free, $R(G)$ is abelian and we are done.
\end{proof}

\begin{lemma}\label{final}
If either $u=2$ and $|\pi(u^\alpha+1)|\neq |\pi(u^\alpha-1)|$, or $u$ is odd, then $G$ does not exist.
\end{lemma}

\begin{proof}
By Lemma \ref{chpsl}, one of the following cases occurs:\\
Case 1:  Either $|\pi(u^\alpha\pm 1)|=n-2$ and $u$ is odd, or for some $\epsilon \in \{\pm 1\}$, $|\pi(u^\alpha +\epsilon)|=n-3$ and $|\pi(u^\alpha -\epsilon)|=n-2$:
Since $|\rho(G)|=2n-1$, $|\pi_F|=3$. Thus by Lemma \ref{free} (c), $\Delta(G)[\pi_F]$ is a non-empty graph. Hence as $|\pi(S)|=2n-4$, using Lemma \ref{easy}, we are done.\\
Case 2: for some $\epsilon\in \{\pm 1\}$, $|\pi(u^\alpha-\epsilon)|=n-1$ and either $|\pi(u^\alpha+\epsilon)|\leqslant n-2$ and $u=2$, or $|\pi(u^\alpha+\epsilon)|=n-2\;or\; n-1$ and $u$ is odd:
Since $|\rho(G)|=2n-1$, $\pi_F\neq \emptyset$. Thus using Lemma \ref{free} (d), the induced subgraph of $\Delta(G)$ on $\pi_F\cup \pi(u^\alpha-\epsilon)$ contains a copy of $K_n$ and it is a contradiction.
\end{proof}

\section*{Acknowledgements}
This research was supported in part
by a grant  from School of Mathematics, Institute for Research in Fundamental Sciences (IPM).


\end{document}